\documentclass[11pt]{article}
\usepackage{amsmath, amsthm, amssymb, amsfonts, amscd, graphicx, verbatim}
\newtheorem{theorem}{Theorem}[section]

\newtheorem{corollary}[theorem]{Corollary}
\newtheorem{lemma}[theorem]{Lemma}

\begin{document}

\title{Isospectrality and matrices with concentric circular higher rank numerical ranges}
             \author{Edward Poon\thanks{Department of Mathematics, Embry--Riddle Aeronautical University, 3700 Willow Creed Road, Prescott, AZ 86301, USA; edward.poon@erau.edu}, and
Hugo J. Woerdeman\thanks{Department of Mathematics, Drexel University, 
3141 Chestnut Street, Philadelphia, PA 19104, USA; hjw27@drexel.edu. }}
\date{ }
          \maketitle



\begin{abstract}
We characterize under what conditions $n\times n$ Hermitian matrices $A_1$ and $A_2$ have the property that the spectrum of $\cos t A_1 + \sin t A_2$ is independent of $t$ (thus, the trigonometric pencil $\cos t A_1 + \sin t A_2$ is isospectral). One of the characterizations requires the first $\lceil \frac{n}{2} \rceil$ higher rank numerical ranges of the matrix $A_1+iA_2$ to be circular disks with center 0. Finding the unitary similarity between $\cos t A_1 + \sin t A_2$ and, say, $A_1$ involves finding a solution to Lax's equation.
\end{abstract}



\section{Introduction}
\label{sec:intro}

Questions regarding rotational symmetry of the classical numerical range as well as the $C-$numerical range have been studied in \cite{ChienTam, Dirretal, Matache, LiSze, LiTsing}; there is a natural connection with isospectral properties. In this paper we study the one parameter pencil ${\rm Re} (e^{-it}B) = \cos t A_1 + \sin t A_2$, where 
$A_1={\rm Re} B = \frac12 (B+B^*) $ and $A_2 = \frac{1}{2i} (B-B^*)$. We say that the pencil is {\it isospectral} when  the spectrum $\sigma ({\rm Re} (e^{it}B))$ of ${\rm Re} (e^{it}B)$ is independent of $t \in [0,2\pi )$; recall that the spectrum of a square matrix is the multiset of its eigenvalues, counting algebraic multiplicity. As our main result (Theorem \ref{thm1}) shows there is a natural connection between isospectrality and the rotational symmetry of the higher rank numerical ranges of $B$. 

Recall that the {\it rank-k numerical range} of a square matrix $B$ is defined by
$$ \Lambda_k (B) = \{ \lambda \in {\mathbb C} : PBP=\lambda P \ {\rm for \ some \ rank}\ k \ {\rm orthogonal \ projection} \ P \} . $$
This notion, which generalizes the classical numerical range when $k=1$ and is motivated by the study of quantum error correction, was introduced in \cite{Choietal1}. In \cite{Choietal2, W} it was shown that $\Lambda_k (B)$ is convex. Subsequently, in \cite{LiSze} a different proof of convexity was given by showing the equivalence 
\begin{equation}\label{CK} z \in \Lambda_k (B) \ \Leftrightarrow \ {\rm Re}(e^{-it}z) \le \lambda_{k} ({\rm Re}(e^{-it}B))\ {\rm for\  all\ } t\in[0,2\pi).\end{equation} Here $\lambda_k (A)$ denotes the $k$th largest eigenvalue of a Hermitian matrix $A$.

In order to state our main result, we consider words $w$ in two letters. For instance, $PPQ$, $PQPQPP$ are words in the letters $P$ and $Q$. The length of a word $w$ is denoted by $|w|$. When we write ${\rm na}(w,P)=l$ we mean that $P$ appears $l$ times in the word $w$ (na=number of appearances).  The trace of a square matrix $A$ is denoted by  ${\rm Tr}\ A$.
\begin{theorem}\label{thm1} Let $B\in{\mathbb C}^{n\times n}$. The following are equivalent.
\begin{itemize}
\item[(i)] The pencil ${\rm Re}(e^{-it}B) = \cos t \ {\rm Re} B +\sin t\ {\rm Im} B$ is isospectral.
\item[(ii)] $\sum_{ |w|=k, {\rm na}(w,B^*)=l}  {\rm Tr}\ w(B,B^*) = 0$, $0\le l < \frac{k}{2}$, $1\le k \le n$.
\item[(iii)] For $1 \leq k \leq \lceil n/2 \rceil$ the rank-$k$ numerical range of $B$ is a circular disk with center $0$, and ${\rm rank} \, {\rm Re} (e^{-it}B)$ is independent of $t$.
\item[(iv)] ${\rm Re}(e^{-it}B)$ is unitarily similar to ${\rm Re}(B)$ for all $t\in [0,2\pi)$.
\end{itemize}
Any of the conditions (i)-(iv) imply that $B$ is nilpotent.
\end{theorem}
Note that for a given matrix $B$ it is easy to check whether Theorem \ref{thm1}(ii) holds. For instance, when $n=5$ one needs to check that $B$ is nilpotent (or, equivalently, ${\rm Tr} B^k=0$, $k=1,\ldots, 5$) and satisfies
$${\rm Tr} B^2B^*={\rm Tr} B^3B^*={\rm Tr} B^4B^*={\rm Tr} B^3B^{2*}+{\rm Tr} B^2B^{*}BB^{*}=0.$$

The paper is organized as follows. In Section \ref{sec2} we prove our main result. In Section \ref{sec3} we discuss the connection with Lax pairs.

\section{Isospectral paths}\label{sec2}

We will use the following lemma.

\begin{lemma}\label{newton} Let $M(t)\in{\mathbb C}^{n\times n}$ for $t$ ranging in some domain. Then the spectrum $\sigma(M(t))$ is independent of $t$ if and only if ${\rm Tr} M(t)^k$, $k=1,\ldots , n$, are independent of $t$.
\end{lemma}

\begin{proof} The forward direction is trivial. For the other direction, use Newton's identities to see that the first $n$ moments of the zeros of a degree $n$ monic polynomial uniquely determine the coefficients of the polynomial, and thus the zeros of the polynomial. This implies that ${\rm Tr} M(t)^k$, $k=1,\ldots , n$, uniquely determine the eigenvalues of the $n\times n$ matrix. Thus, if ${\rm Tr} M(t)^k$, $k=1,\ldots , n$, are independent of $t$, then the spectrum of $M(t)$ is independent of $t$. \end{proof}

\noindent {\it Proof of Theorem \ref{thm1}.} Consider the trigonometric polynomials $f_k(t)= 2^k {\rm Tr} [{\rm Re}(e^{-it}B)]^k$, $k=1,\ldots , n$. The coefficient of $e^{i(2l-k)t}$ in $f_k(t)$ is given by $\sum_{ |w|=k, {\rm na}(w,B^*)=l} {\rm Tr}\ w(B,B^*)$. By Lemma \ref{newton} the spectrum of ${\rm Re}(e^{-it}B) $ is independent of $t$ if and only for $k=1,\ldots, n$ and $2l\neq k$ the coefficient of $e^{i(2l-k)t}$ in $f_k(t)$ is 0. Due to symmetry, when they are 0 for $2l< k$ they will be 0 for $2l> k$. This gives the equivalence of (i) and (ii).

In particular note that when $l=0$, we find that $ {\rm Tr} B^k =0$, $k=1,\ldots, n$, and thus $B$ is nilpotent.

Next, let us prove the equivalence of (i) and (iii). Assuming (i) we have that ${\rm Re} B$ and $-{\rm Re} B$ have the same spectrum, so ${\rm Re} B$ has $\lceil n/2 \rceil$ nonnegative eigenvalues. As the spectrum of ${\rm Re}(e^{-it}B) $ is independent of $t$, we have that ${\rm Re}(e^{-it}B)$ has $\lceil n/2 \rceil$ nonnegative eigenvalues for all $t$, guaranteeing the rank-$k$ numerical range is nonempty for $k \leq \lceil n/2 \rceil$. Next, since $\lambda_{k} ({\rm Re}(e^{-it}B))$ is independent of $t$, it immediately follows from the characterization \eqref{CK} that $\Lambda_k(B)$, $1 \leq k \leq \lceil n/2 \rceil$, is a circular disk with center 0. Also, (i) clearly implies that ${\rm rank} (e^{-it}B)$ is independent of $t$.

Conversely, let us assume (iii). If the rank $k$-numerical range of $B$ is $\{z : |z| \leq r\}$ for some $r > 0$ then 
$\lambda_{k} ({\rm Re}(e^{-it}B))$ is constant. This also yields that $\lambda_{n+1-k}({\rm Re}(e^{-it}B)) = -\lambda_k(-{\rm Re}(e^{-it}B))$. When for $1 \leq k \leq \lceil n/2 \rceil$ we have that $\Lambda_k (B)$ has a positive radius, we obtain that (i) holds. Next, let us suppose $\Lambda_{\ell} (B)$ has radius zero, and $\ell$ is the least integer with this property. 
Then, as before, we may conclude that $\lambda_{k} ({\rm Re}(e^{-it}B))$ is a positive constant for $1\leq k <\ell$. We also have, for $\ell \leq k \leq \lceil n/2 \rceil$, that $\lambda_{k} ({\rm Re}(e^{-it}B)) = 0 $ for some $t$. As we require ${\rm rank} \, {\rm Re}(e^{-it}B)$ to be independent of $t$, we find that for $\ell \leq k \leq \lceil n/2 \rceil$, $\lambda_{k} ({\rm Re}(e^{-it}B)) = 0 $ for all $t$. Again using $\lambda_{n+1-k}({\rm Re}(e^{-it}B)) = -\lambda_k(-{\rm Re}(e^{-it}B))$, we arrive at (i).

The equivalence of (i) and (iv) is obvious.
\hfill $\square$

\medskip

\noindent {\bf Remark.} The condition that ${\rm rank} \, {\rm Re}(e^{-it}B)$ is independent of $t$ in Theorem \ref{thm1}(iii) is there to handle the case when $\Lambda_k (B)$ has a zero radius. Indeed, it can happen that $\Lambda_k (B) =\{ 0 \}$ without $\lambda_{k} ({\rm Re}(e^{-it}B))$ being independent of $t$; one such example is a diagonal matrix with eigenvalues $1,0,-1,i$. It is unclear whether this can happen for a matrix whose higher rank numerical ranges are disks centered at 0. 

\bigskip

For sizes 2, 3, and 4, the conditions in Theorem \ref{thm1} are equivalent to $B$ being nilpotent and the numerical range of $B$ being rotationally symmetric.

\begin{corollary}\label{cor} Let $B \in {\mathbb C}^{n\times n}$, $n\le 4$. Then the spectrum of ${\rm Re}(e^{-it}B) = \cos t \ {\rm Re} B +\sin t\ {\rm Im} B$ is independent of $t$  if and only if $B$ is nilpotent and the numerical range is a disk centered at 0.
\end{corollary}

\begin{proof} When $n=2$, condition (ii) in Theorem \ref{thm1} comes down to ${\rm Tr} B={\rm Tr} B^2=0$. When $n=3$ we get the added conditions that ${\rm Tr} B^3={\rm Tr} B^2B^*=0$. When $n=4$, we also need to add the conditions ${\rm Tr} B^4={\rm Tr} B^3B^*=0$. The condition that  ${\rm Tr} B^k=0$, $1\le k \le n$, is equivalent to $B$ being nilpotent. The corollary now easily follows by invoking Remarks 1-3 in \cite{Matache}. \end{proof} 

To show that Corollary \ref{cor} does not hold for $n\ge 5$, note that the following example from \cite{Matache},
$$ B= \begin{pmatrix} 0 & 2 & 0 & 0 & 0 \cr 0 & 0 & 0 & 0 &0 \cr 0 & 0 & 0 & 1 & 1 \cr 0 & 0 & 0 & 0 & 1 \cr 0 & 0 & 0 & 0 & 0 \end{pmatrix},$$
is nilpotent, has the unit disk as its numerical range, but ${\rm Tr} B^2B^*=1\neq 0$.

\section{Connection with Lax pairs}\label{sec3}

A {\it Lax pair} is a pair $
L(t),P(t) $ of Hilbert space operator valued functions
 satisfying Lax's equation:
$$ {\frac  {dL}{dt}}=[P,L], $$ where $[X,Y] = XY-YX$. The notion of Lax pairs goes back to \cite{Lax}. If we start with $P(t)$, and one solves the initial value differential equation 
\begin{equation}\label{Ut} {\frac  {d}{dt}}U(t)=P(t)U(t),\qquad U(0)=I, \end{equation}
then $L(t):=U(t)L(0)U(t)^{{-1}}$ is a solution to Lax's equation. Indeed,
$$ L'(t) = {\frac  {d}{dt}} [U(t)L(0)U(t)^{{-1}}]= $$ $$ P(t) U(t) L(0) U(t)^{{-1}} - U(t) L(0) U(t)^{{-1}} P(t) U(t) U(t)^{-1}  
= P(t)L(t) - L(t) P(t) . $$
This now yields that $L(t)$ is isospectral.
When $P(t)$ is skew-adjoint, then $U(t)$ is unitary.

In our case we have that $L(t) = {\rm Re} (e^{-it} B )$, and our $U(t)$ will be unitary. This corresponds to $P(t)$ being skew-adjoint. When we are interested in the case when  $P(t) \equiv K$ is constant, we have that $U(t)=e^{tK}$.
Thus, we are interested in finding $K$ so that $e^{-tK} L(t) e^{tK} = L(0)$, where $L(t) = A_1 \cos t + A_2 \sin t$. If we now differentiate both sides, we find 
$$ -e^{-tK} K L(t) e^{tK} + e^{-tK} L'(t) e^{tK} + e^{-tK} L(t) K e^{tK} = 0. $$
Multiplying on the left by $e^{tK}$ and on the right by $e^{-tK}$, we obtain 
$$-A_1 \sin t + A_2 \cos t = L'(t) = [K,L(t)] = [K,A_1 \cos t + A_2 \sin t] . $$
This corresponds to $[K,A_1] = A_2$ and $[K,A_2] = -A_1$, which is equivalent to $[K,B] = -iB$. We address this case in the following result, which is partially due to \cite{LiTsing}. 
\medskip

\begin{theorem}\label{constant2}
Let $B\in {\mathbb C}^{n\times n}$. 
The following are equivalent.
\begin{itemize}
\item[(i)] $e^{it}B$ is unitarily similar to $B$ for all $t\in[0,2\pi)$. 
\item[(ii)] ${\rm Tr}\ w(B,B^*)=0$ for all words $w$ with ${\rm na} (w,B) \neq {\rm na} (w,B^*)$.
\item[(iii)] There exists a skew-adjoint matrix $K $ satisfying $[K,B] = -iB$.
\item[(iv)] There exists a unitary matrix $U$ such that $UBU^* = B_1 \oplus \dots \oplus B_r$ is block diagonal and each submatrix $B_j$ is a partitioned matrix (with square matrices on the block diagonal) whose only nonzero blocks are on the block superdiagonal.   
\end{itemize}
 \end{theorem}

Recall that Specht's theorem \cite{Specht} says that $A$ is unitarily similar to $B$ if and only if ${\rm Tr}\ w(A,A^*) = {\rm Tr}\  w(B,B^*)$ for all words $w$.

\begin{proof} By Specht's theorem  $e^{it}B$ is unitarily similar to $B$ for all $t$ if and only if ${\rm Tr}\ w(e^{it}B,e^{-it}B^*) = {\rm Tr}\ w(B,B^*)$ for all $t$ and all words. When ${\rm na} (w,B) \neq {\rm na} (w,B^*)$ this can only happen when ${\rm Tr}\ w(B,B^*)=0$. When ${\rm na} (w,B) = {\rm na} (w,B^*)$, we have that ${\rm Tr}\ w(e^{it}B,e^{-it}B^*)$ is automatically independent of $t$. This proves the equivalence of (i) and (ii).

The equivalence of (i) and (iv) is proven in \cite[Theorem 2.1]{LiTsing}. 
We will finish the proof by proving (iv) $\to$ (iii) $\to$ (i). 

Assuming (iv), let $K_j$ be a block diagonal matrix partitioned in the same manner as $B_j$ and whose $m$th diagonal block equals $imI$.  Then $[K_j,B_j] = -i B_j$.  Let $K = U^*(K_1 \oplus \dots \oplus K_r)U$.  Then $[K,B] = -iB$,
proving (iii).

When (iii) holds, let $U(t) = e^{-Kt}$. 
Denote ${\rm ad}_X Y = [X,Y]$. Then $e^X Y e^{-X} = \sum_{m=0}^\infty \frac{1}{m!} {\rm ad}_X^m Y$, and (iii) yields that 
$$ U(t) B U(t)^* = e^{-tK} B e^{tK} = \sum_{m=0}^{\infty} \frac{1}{m!} {\rm ad}_{-tK}^m B = \sum_{m=0}^{\infty} \frac{(it)^m}{m!} B = e^{it} B,$$
yielding (i).
\end{proof}

It is clear that if $B$ satisfies Theorem \ref{constant2}(i) it certainly satisfies Theorem \ref{thm1}(i). In general the converse will not be true, and the size of such a counterexample must be at least 4; indeed, if $B$ is a strictly upper triangular $3\times 3$ matrix with ${\rm Tr} B^2B^*=0$ at least one of the entries above the diagonal is zero, making $B$ satisfy Theorem \ref{constant2}(iv).
An example that satisfies the conditions of Theorem \ref{thm1} but does not satisfy those of Theorem \ref{constant2} is
\begin{equation}\label{B} B=\begin{pmatrix} 0 & 1 & 1 &0 \cr 0 & 0 & 1 & -1  \cr 0 & 0 & 0 & 1  \cr 0 &0&0&0 \end{pmatrix}. 
\end{equation}
Indeed, it is easy to check that 
${\rm Tr} B^2B^*={\rm Tr} B^3B^*=0$, but ${\rm Tr} B^3 B^* B B^* =-1 \neq 0$.  A $5 \times 5$ example satisfying the conditions of  
Theorem \ref{thm1} but not those of Theorem \ref{constant2} is
\begin{equation*}
	\begin{pmatrix} 0 & 1 & 1/2 & 1 & 0 \\ 0 & 0 & 1 & -1 & -1 \\ 0 & 0 & 0  & 1 & 3/2 \\ 0 & 0 & 0 & 0 & 1 \\ 0 & 0 & 0 & 0 & 0 \end{pmatrix}.
\end{equation*}

When $B$ satisfies the conditions of Theorem \ref{constant2}, the $K$ from Theorem \ref{constant2}(iii) will yield the unitary similarity ${\rm Re} (e^{it}B) = e^{-tK} ({\rm Re} B) e^{tK}$. It is easy to find $K=-K^*$ satisfying $[K,B]=-iB$
as it amounts to solving a system of linear equations (with the unknowns the entries in the lower triangular part of $K$).

When $B$ satisfies the conditions of Theorem \ref{thm1}, but not those of Theorem \ref{constant2}, finding a unitary similarity $U(t)$ so that ${\rm Re} (e^{-it}B) = U(t) ({\rm Re} B) U(t)^*$ becomes much more involved. To go about this one could first find a solution $P(t)$ to Lax's equation
$$-A_1 \sin t + A_2 \cos t = L'(t) = [P(t),L(t)] = [P(t),A_1 \cos t + A_2 \sin t],   $$ which now will not be constant. Next, one would solve the initial value ordinary differential matrix equation \eqref{Ut}. 

To illustrate what a solution $P(t), U(t)$ may look like, we used Matlab to produce the following solution when 
$A_1={\rm Re} \ B$ and $A_2={\rm Im} \ B$ (and thus $L(t)={\rm Re}(e^{-it} B)$) with $B$ as in \eqref{B}: 

$$ 
P(t) = \begin{pmatrix} -\frac{i}{2} & 0 & 0 & \frac{ie^{-2it}}{2} \cr 0 & 0 & 0 & 0 \cr 0 & 0 & i & 0 \cr \frac{ie^{2it}}{2} & 0 & 0 & \frac{3i}{2} \end{pmatrix}, 
$$
$$ V(t) = \begin{pmatrix} 1-e^{-it} & -1-e^{-it} & 1-e^{-it} & 1+e^{-it} \cr 2 & 1 & -1 & 2 \cr 
-2e^{it} & e^{it} & e^{it} & 2e^{it} \cr e^{2it}+e^{it} & -e^{2it}+e^{it} & e^{2it}+e^{it} & e^{2it}-e^{it} \end{pmatrix} , U(t)= V(t)V(0)^{-1}. $$
Note that the columns of $V(t)$ are the eigenvectors of $L(t)$; indeed, we have $$L(t) = V(t) {\rm diag}(-1,-\frac12,\frac12,1) V(t)^{-1}.$$ 


%
%

\section{Acknowledgments}

The research of Hugo J. Woerdeman was supported by Simons Foundation grant 355645 and National Science Foundation grant DMS 2000037.

\bibliographystyle{plain}
 
\end{document}